\documentclass[12pt,reqno]{amsart}
\usepackage{fullpage}

\usepackage[utf8]{inputenc}
\usepackage{amsmath}
\usepackage{amssymb}
\usepackage{amsthm}
\usepackage{graphicx}
\usepackage[dvipsnames]{xcolor}
\usepackage[hyphens]{url}
\usepackage[hidelinks, breaklinks=true]{hyperref}
\usepackage[hyphenbreaks]{breakurl}
\usepackage[backend=biber, style=alphabetic, url=false, doi=false, isbn=false,firstinits=true]{biblatex}
\usepackage{mathtools}
\usepackage{microtype} 
\usepackage{enumitem}

\DeclareMathOperator{\Gal}{Gal}
\DeclareMathOperator{\tors}{tors}

\DeclareMathOperator{\Disc}{disc}

\newcommand{\Z}{\mathbb{Z}}
\newcommand{\N}{\mathbb{N}}
\newcommand{\Q}{\mathbb{Q}}	
\newcommand{\R}{\mathbb{R}}
\newcommand{\C}{\mathbb{C}}
\newcommand{\p}{\mathfrak{p}}
\newcommand{\q}{\mathfrak{q}}
\newcommand{\rank}{	\mathrm{rank}}
\newcommand{\D}{\mathcal{D}}
\newcommand{\F}{\mathcal{F}}

\newcommand{\Ektors}{E(k)_{\tors}}

\newcommand{\OK}{\mathcal{O}_{K}}
\newcommand{\Ok}{\mathcal{O}_{k}}
\newcommand{\NkQideal}{N}
\newcommand{\NkQ}{N_{k/\Q}}
\newcommand{\Inertiap}{I_{\hat{\mathfrak{P}}}}

\newtheorem*{acknowledgements*}{Acknowledgements}
\theoremstyle{plain}
\newtheorem{theorem}{Theorem}[section]
\newtheorem{prop}[theorem]{Proposition}
\newtheorem{lemma}[theorem]{Lemma}

\newtheorem{corollary}[theorem]{Corollary}
\newtheorem{remark}[theorem]{Remark}
\DeclarePairedDelimiter\abs{\lvert}{\rvert}%
\DeclarePairedDelimiter\norm{\lVert}{\rVert}%
\makeatletter
\let\oldabs\abs
\def\abs{\@ifstar{\oldabs}{\oldabs*}}
\let\oldnorm\norm
\def\norm{\@ifstar{\oldnorm}{\oldnorm*}}
\makeatother

\title{Uniform bounds for the number of rational points of bounded height on
	certain elliptic curves}
\author{Marta Dujella}
\address[Marta Dujella]{University of Basel, Spiegelgasse 1, 4051 Basel, Switzerland.}\email{marta.dujella@unibas.ch}
\date{}

\begin{document}
	\maketitle
	\begin{abstract}
		Let $E$ be an elliptic curve defined over a number field $k$ and $\ell$ a prime integer. When $E$ has at least one $k$-rational point of exact order $\ell$, we derive a uniform upper bound $\exp(C \log B / \log \log B)$ for the number of points of $E(k)$ of (exponential) height at most $B$. Here the constant $C = C(k)$ depends on the number field $k$ and is effective. For $\ell = 2$ this generalizes a result of Naccarato \cite{Naccarato:2021Counting} which applies for $k=\Q$. We follow methods previously developed by Bombieri and Zannier \cite{Bombieri:2004On-the-number} and Naccarato \cite{Naccarato:2021Counting}, with the main novelty being the application of Rosen's result on bounding $\ell$-ranks of class groups in certain extensions, which is derived using relative genus theory.
	\end{abstract}
	\maketitle
	\section{Introduction} \label{Introduction}
	Let $k$ be a number field of degree  $d = [k:\Q]$ and $E$ an elliptic curve defined over $k$ by a Weierstrass equation
	\begin{equation}
	\label{weierstrass_equation}
	E: y^2=F(x) = x^3+ax^2+bx+c, \quad a,b \in k.
	\end{equation}
	For a point $P$ in a projective space, let $H(P)$ and $h(P)=\log H(P)$ denote the absolute exponential and logarithmic Weil heights, respectively. By $h(E)$ we will denote the height of the equation \eqref{weierstrass_equation} defining $E$, i.e. $h(E) := h(1:a:b:c)$, the logarithmic Weil height of a projective point $(1:a:b:c) \in \mathbb{P}^3_{k}$. Furthermore, let $\hat{h}$ be the canonical height on $E$. The precise definitions for heights used are given in Section \ref{Definitions_strategy}.
	
	For a real number $B > 0$ we denote by
	\begin{equation*}
		\mathcal{N}_E (B) := \abs{ \left\{ P \in E(k) : \hat{h}(P) \leq \log B \right\}}.
	\end{equation*}
	It is a classical result of N\'{e}ron that $ \mathcal{N}_E (B) \sim c_E (\log B)^{r/2}$, as $B \rightarrow \infty$, where $r$ is the rank of $E(k)$, which is always finite due to Mordell--Weil theorem. For applications it is often useful to have a bound on $\mathcal{N}_E (B)$ that is uniform in $E$ and $B$, see for example Corollary in \cite{Bombieri:2004On-the-number} and Sections 4 and 5 in \cite{Bombieri:2015A-problem} . The first notable result on this topic was by Bombieri and Zannier \cite{Bombieri:2004On-the-number}, who looked at elliptic curves $E$ defined over $\Q$ with $E[2] \subset E(\Q)$, where $E[2]$ denotes the subgroup of  $2$-torsion points of $E$. They showed that for these curves and for $B$ sufficiently large in terms of $E$ one has
	$
	\mathcal{N}_E (B) \leq B^{C/\sqrt{\log \log B}}
	$
	with $C$ being an absolute constant. Furthermore, in section 2 of \cite{Bombieri:2004On-the-number}, they note that the assumption $E[2] \subset \Q$ can be weakened to $E[2] \cap E(\Q) \neq \{O\}$ by performing descent via 2-isogeny. Under this weaker hypothesis, Naccarato \cite{Naccarato:2021Counting} proved a better asymptotic bound $\mathcal{N}_E (B) \leq B^{C/\log \log B}$. The main ingredient of this improvement was using a different lower bound for $\hat{h}$ of nontorsion points.
	
	In this paper we generalize this result to arbitrary number fields. Additionally, we get a bound of the same quality for the families of elliptic curves with a non-trivial $k$-rational $\ell$-torsion point, for any prime integer $\ell$.
	\begin{theorem}
		Let $k$ be a number field, $\ell$ a prime integer and $E/k$ an elliptic curve given by \eqref{weierstrass_equation}. If $E(k) \cap E[\ell] \neq \{O\}$, i.e. $E$ has  at least one point of exact order $\ell$ over $k$, then
		\label{main_thm_can}
		\[\abs{\left\{ P \in E(k) : \hat{h}(P) \leq \log B \right\}} \leq B^{C/ \log \log B}\]
		for all $B \geq e^{\max \left\{h(E), e \right\}}$, where $C = C(k,\ell)$ is an effective constant depending on the field $k$ and the prime $\ell$.
	\end{theorem}
	Here the constant $C$ in general depends on the number field $k$. It is worth noting that if we start out with an integral model for $E$, i.e. if $a, b,c$ in \eqref{weierstrass_equation} are in $\Ok$, then we get a constant $C$ that depends only on the degree $d$ of the number field and the $\ell$-part of its class group: $\rank_{\Z / \ell \Z}Cl_k[\ell]$. For this we need to additionally assume that $k$ contains a primitive $\ell$-th root of unity. Notice that for $\ell=2$ this is always trivially satisfied.
	\begin{theorem}
		Let $k$ be a number field with $d = [k: \Q]$ and $\ell$ and $E/k$ as in Theorem \ref{main_thm_can}. Assume further that $a,b,c \in \Ok$. If $\ell \geq 3$ we also assume that $k$ contains a primitive $\ell$-th root of unity. Then
		\label{main_thm_integral}
		\[\abs{\left\{ P \in E(k) : \hat{h}(P) \leq \log B \right\}} \leq B^{C'/ \log \log B}\]
		for all $B \geq e^{\max \left\{h(E), e \right\}}$, where $C' = C'(d, \ell, \rank_{\Z / \ell \Z} Cl_k[\ell])$ is an effective constant depending only on $d, \ell$ and $\rank_{\Z / \ell \Z} Cl_k[\ell])$.
	\end{theorem}
	
	Let us now briefly summarize our approach and give an overview of the contents of this paper. We follow a method developed by Bombieri and Zannier \cite{Bombieri:2004On-the-number}, at the core of which is the following observation: either $E$ has few places of bad reduction over $k$, in which case we get a good bound for the rank, or $E$ has many places of bad reduction, which we can then exploit to get a good lower bound for the height. In Section \ref{Definitions_strategy}, we give an overview of the counting strategy used as well as definitions of heights used throughout the paper. We use a classical lattice point counting approach in $E(k) \otimes \R $. Next, Section \ref{Lower_bound_height} mostly follows work of Naccarato and uses a result of Petsche \cite{Petsche:2006Small} on a lower bound for $\hat{h}(P)$ to bound $\mathcal{N}_E (B)$ in terms of  the rank and the conductor of $E/k$. Additionally, we have to deal with curves with everywhere good reduction, for which we use another lower bound for $\hat{h}(P)$ due to David \cite{David:1997Points}. 
	
	Section \ref{Bounding_rank} concerns deriving a suitable bound for the Mordell--Weil rank, in terms of the number of places of bad reduction of $E/k$. This part contains most new insights. We look at the extension $K_\ell:=k(E[\ell]) \supset k$, with $\ell$ a prime integer. Over this field $E$ has full $\ell$-torsion, so the problem reduces to relating $\rank_{\Z / \ell \Z}Cl_{K_\ell}[\ell]$ to $k$. Bounding torsion in class groups is in general a difficult problem. For solving this we use a result of Rosen \cite{Rosen:2014Class} on bounding $\ell$-ranks of class groups in certain extensions. This reduces the problem to studying ramified primes in extension $K_\ell \supset k$, which we do by using a N\'{e}ron--Ogg--Shafarevich type result.
	Let us remark that Rosen's result on class groups is a generalization of classical Gauss genus theory that describes $Cl_K[2]$ when $K$ is an imaginary quadratic field. Gauss showed that if the discriminant of $K$ has precisely $t$ distinct prime divisors, then $Cl_K[2]$ has order $2^{t-1}$. 
	
	To conclude this section, we show that Theorem \ref{main_thm_can} easily implies the following corollary.  For $P = [x:y:z] \in \mathbb{P}^2(k)$ we write  $x(P) = [x:z]$.
	\begin{corollary}
		Let $k$ be a number field and $E/k$ an elliptic curve as in Theorem \ref{main_thm_can}. Let $B$ be a real number with $B \geq e^{\max \left\{h(E), e\right\}}$. Then
		\[ \abs{\left\{ P \in E(k) : H(x(P)) \leq B \right\}} \leq B^{C''/ \log \log B},\]
		where $C''=C''(k)$ is a constant depending on the field $k$.
	\end{corollary}
	\begin{proof}
		It is known that the difference between the canonical height $\hat{h}(P)$ and $h(x(P)) = \log H(x(P))$ is bounded. For example from Remark 1.2 in \cite{Silverman:1990The-difference} we have that there exist constants $C_1, C_2 \geq 1$ such that
		\begin{equation*}
			\hat{h}(P) \leq \frac{1}{2}h(x(P)) +C_1 h(E) + C_2 \leq h(x(P)) +C_1 h(E) + C_2.
		\end{equation*}
		If we now assume that $H(x(P)) \leq B$ for some $P \in E(K)$ we see that
		\begin{equation*}
			P \in \left\{ Q \in E(k) : \hat{h}(Q) \leq \log B' \right\},
		\end{equation*}
		where $\log B' := \log B + C_1 h(E) +C_2$. Because $B' \geq B \geq e^{\max \left\{h(E),e \right\}}$, we can apply Theorem \ref{main_thm_can} to get
		\begin{align*}
			 \log \left(\abs{\left\{ P \in E(k) : H(x(P)) \leq B \right\} }\right) &\leq C \frac{\log B'}{\log \log B} = C\frac{\log B + C_1 h(E)+C_2}{\log \log B'} \\
				& \leq C\left(1+C_1 h(E)+\frac{C_2}{e}\right)\frac{\log B}{\log \log B} ,
		\end{align*}
		because $\log B \geq \max \left\{h(E),e \right\}$ and $\log B' \geq \log B$.
	\end{proof}
	
	\begin{acknowledgements*}
		\normalfont
		The author would like to thank Philipp Habegger for introducing her to the problem and many helpful discussions. Additionally, the author was supported by the Swiss National Science Foundation grant	“Diophantine Equations: Special Points, Integrality, and Beyond” (no.$200020\_184623$)."
	\end{acknowledgements*}
	
	\section{Definitions and the counting strategy} \label{Definitions_strategy}
	In this section we review some of the relevant definitions, mainly concerning heights on projective spaces and elliptic curves over number fields, as well as fix some notations that will be used throughout the article.
	
	Let $k$ be a number field and $\Ok$ its ring of integers. We write $\NkQ(x)$ for the norm of an element $x \in k$ and $\NkQideal(I)$ for the norm of an ideal $I$ of $\Ok$.
	
	By $M_k^0$ and $M_k^{\infty}$ we denote the set of all finite and infinite places of $k$ respectively, and write $M_k$ for $M_k^0 \cup M_k^{\infty}$. 
	We identify $M_k^0$ with the set of maximal ideals of $\Ok$ and normalize the absolute values as follows:
	For $\p$ a maximal ideal of $\Ok$ we define $\abs{x}_{\p} := p(\p)^{-\nu _{\p}(x) / e(\p)}$, with $p(\p)$ being the unique prime of $\N$ lying below $\p$, $e(\p)$ the ramification index of $P$ and $\nu_{\p}$ the valuation on $k$ attached to $\p$ with image $\Z \cup \left\{\infty\right\}$. Furthermore, by $f(\p)$ we denote the inertia degree of $\p$. The set $M_k^{\infty}$ of all infinite places of $k$ can be identified with the set of field embeddings $k \hookrightarrow \C$ up to complex conjugation. So, for $v  \in M_k^{\infty}$ there is $\sigma : k \hookrightarrow \C $ with $\abs{x}_v = \abs{\sigma(x)}$ for all $x \in k$.
	
	For $v = \p \in M_k^0$ we set $d_v = [k_v : \Q_p]$, where $k_v$ and $\Q_p$ are completions of $k$ and $\Q$ with respect to $v$ and $p = p(\p)$ respectively.  For $v \in M_k^{\infty}$ we set $d_v=1$ if $v$ corresponds to a real embedding, and $d_v=2$ otherwise. With this we can define the absolute (logarithmic) height. Let $n \in \N$ and $x = [x_0 : \dots : x_n] \in \mathbb{P}^n(k)$ with $x_0, \ldots, x_n \in k $, then the height of $x$ is
	\begin{align*}
		h(x) = \frac{1}{[k:\Q]}\sum_{v \in M_k} d_v \log \max \{\abs{x_0}_v, \dots, \abs{x_n}_v\}.
	\end{align*}
 	For $x = (x_1, \ldots, x_n) \in k^n$ we define the affine height of $x$ as $h(x) = h(1:x_0: \ldots :x_n)$, i.e. the projective height of $[1:x_0: \ldots : x_n]$. Furthermore, for $x \in k$ we define the (logarithmic) height as $h(x) = h([1:x])$. The exponential height is $H(x) = \exp(h(x))$, for $x \in \mathbb{P}^n(k)$, $k^n$ or $k$.
	
	Let now $E$ be an elliptic curve defined over $k$, given by a Weierstrass equation \eqref{weierstrass_equation}. We define the (naive) height of $E$ as $h(E) = h(a,b,c)$. For $P = (x,y)  \in E(k) \setminus \{O\}$, put $h_x(P) = h(x(P)) = h(x)$. Then the canonical (or N\'{e}ron-Tate) height of $P \in E(k) \setminus \{O\}$ is given by
	\begin{align*}
		\hat{h}(P) = \frac{1}{2} \lim_{n \to \infty} \frac{h_x([2^n]P)}{2^{2n}},
	\end{align*}
	while $\hat{h}(O) = 0$. 
	This defines a quadratic form on $E(k)$ that is zero exactly when $P$ is a torsion point.
	Moreover, if  $E'/k$ is another elliptic curve and $\phi: E \rightarrow E'$ an isomorphism defined over $k$, then $\hat{h}_E(P) = \hat{h}_{E'}(\phi(P))$ for all $P \in E(k)$.
	
	We denote the minimal discriminant and the conductor of $E$ over $k$ by $\D_{E/k	}$ and $\F_{E/k}$ respectively. For definitions see for example VIII.8 in \cite{Silverman:2009The-arithmetic} and IV.10 in \cite{Silverman:1994Advanced} ; let us remark here that both of those are non-zero integral ideals of $\Ok$ that are supported exactly on primes where $E$ has bad reduction. We will also need the Szpiro ratio of $E$ over $k$, defined as
	\begin{equation*}
		\label{szpiro_def}
		\sigma_{E/k} = \frac{\log \NkQideal (\D_{E/k})}{\log \NkQideal (\F_{E/k})}
	\end{equation*}
	when $E$ has at least one place of bad reduction over $k$. In the case of everywhere good reduction we put $\sigma_{E/k} = 1$. It is a well known fact (see Corollary 11.2 in \cite{Silverman:1994Advanced}) that $\nu_{P}(\D_{E/k}) \geq \nu_{P}(\F_{E/k})$ for all $P \in M_k^0$ and consequently $\sigma_{E/k} \geq 1$.
	
	By the Mordell--Weil theorem we know that $E(k) \cong \Ektors \times \Z^r$ for some  nonnegative integer $r = \rank E(k)$. This implies $E(k) \otimes  \R \cong \R^r$ and the image of $E(k)/\Ektors$ is a lattice $\Lambda_E$ of rank $r$ in this space. As explained in \cite{Bombieri:2006Heights}, Chapter 9.3., we can define a scalar product on $E(k) \otimes \R$ such that $\langle P \otimes 1 , P \otimes 1 \rangle = \hat{h}(P)$ and with this it becomes a Euclidean space with a norm $ \norm{x}  := \langle x,x \rangle ^{1/2}$. So if $P$ is a point in $E(k)$ we have that $ \norm{\overline{P}}  = \hat{h}(P)^{1/2}$, where $\overline{P}$ denotes the class of $P$ modulo $\Ektors$.
	
	Now we can see that counting the points of $E(k)$ with the canonical height at most $\log B$ is equivalent to counting the number of points of $\Lambda_E$ that are contained in the ball of radius $\sqrt{\log B}$ around $0$ in $\R^r$ (equipped with the norm $ \norm{\cdot} $  described above)  and multiplying  by  the size of $\Ektors$. We will write $B(x,R)$ and $\overline{B}(x,R)$ for an open and a closed ball around $x$ of radius $R$ in $(\R^r,\norm{\cdot} )$ respectively.
	
	We follow the counting strategy that was used by Naccarato in  \cite{Naccarato:2021Counting} and is itself  an adaptation of one used by Bombieri and Zannier in \cite{Bombieri:2004On-the-number}. The idea is to find $\rho >0 $ such that we are sure that any ball of radius $\rho$ centered at a point of $\Lambda_E$ contains at most (so exactly) one point of that lattice, in other words $\overline{B}(x,\rho) \cap \Lambda_E = \left\{ x \right\}$ for all $x \in \Lambda_E $. Then we just need to count how many such balls we need to cover the $\overline{B}(0,\sqrt{\log B}) \cap \Lambda_E$. This second part is done using a standard covering argument (see for example Lemme 6.1 in \cite{Remond:2000Decompte}) which we quote here. 	Let us remark that while the statement of the lemma in \cite{Remond:2000Decompte} is given only for $R \geq \rho$ and closed balls, the proof remains valid for all $R, \rho >0$ and open balls.
	\begin{lemma}
		\label{balls}
		Let $A$ be contained in some  closed ball of radius $R>0$ in $\R^r$ and let $\rho >0$. Then $A$ can be covered by at most $(1+\frac{2R}{\rho})^r$ open balls of radius $\rho$ centered at points of $A$.
	\end{lemma}
	
	It is a well known fact (see for example \cite{Silverman:2009The-arithmetic}, Chapter VIII.9) that the torsion points $P$ in $E(k)$ are characterized as exactly those points with $\hat{h}(P)=0$. Hence a lower bound for $\hat{h}(P)$ for all nontorsion points will give us a suitable $\rho$ for the strategy described above: if for some $c>0$ we have that $\hat{h}(P) \geq c$ for all $P \in E(k) \setminus \Ektors$, then $B(x, \sqrt{c}) \cap \Lambda_E = \left\{ x \right\}$ for every $x \in \Lambda_E$.
	
	In what follows $c_1, c_2, \ldots $ always denote positive effective constants, where $c_i(d)$  and $c_i(k)$ means that the corresponding constant depends on $d=[k:\Q]$ or the field $k$, respectively. For simplicity of the notation, the dependence of the constants on $\ell$, which appears from section \ref{Bounding_rank} onwards, is implied.
	
	\section{Lower bound on the height} \label{Lower_bound_height}
	Let us recall that $k$ is a fixed number field with $d = [k: \Q]$ and $E/k$ is given by \eqref{weierstrass_equation}. It will be useful to work with a model for $E$ whose Weierstrass coefficients  are in $\Ok$, but that is still not "too big" in a certain sense, when compared to our starting equation defining $E$. The discriminant of $E$ is given by $\Delta_{E} = 16 \Disc (F) =  - 16 \left(-a^2 b^2 + 4 a^3 c + 4 b^3 - 18 a b c + 27 c^2\right)$ and we are searching for $E' \cong E$ whose discriminant is bounded in terms of $h(E) = h(a,b,c)$. We find such an equation for $E$ in the following two lemmas. For $a,b,c \in k$ we introduce the notation for the finite and infinite parts of $h(a,b,c)$:
	\begin{align*}
		h^{0}(a,b,c) &=\frac{1}{[k:\Q]} \sum_{v \in M_k^{0}} d_v \log \max \{1,\abs{a}_v,\abs{b}_v, \abs{c}_v\}, \\
		h^{\infty}(a,b,c) &=\frac{1}{[k:\Q]} \sum_{v \in M_k^{\infty}} d_v \log \max \{1,\abs{a}_v,\abs{b}_v,\abs{c}_v\}.
	\end{align*}
	Therefore, $h(a,b,c) = h^0(a,b,c)+h^{\infty}(a,b,c)$ and we remark that if $a,b,c \in \Ok$, then $h(a,b,c) = h^{\infty}(a,b,c)$.
	\begin{lemma}
		\label{model_lemma1}
		Let $E/k$ be an elliptic curve as in \eqref{weierstrass_equation} with discriminant $\Delta_{E}$. Then
		\begin{align*}
			\log \abs{\NkQ (\Delta_{E})} \leq d \log(2^4 \cdot 54) + 4d h^{\infty}(a,b,c).
		\end{align*}
		In particular, if $a,b,c \in \Ok$, we have
		\begin{align*}
			\log \abs{\NkQ (\Delta_{E})} \leq d \log(2^4 \cdot 54) + 4d h(E).
		\end{align*}
	\end{lemma}
	\begin{proof}
		For $v \in M_k^\infty$, let $\mu_v := \max \left\{1, \abs{a}_v, \abs{b}_v,\abs{c_v}\right\}$. Then 
		\begin{equation*}
			\abs{\Delta_{E}}_v \leq 16 \left(\abs{a}_v^2 \abs{b}_v^2 + 4 \abs{a}_v^3 \abs{c}_v + 4 \abs{b}_v^3 + 18 \abs{a}_v \abs{b}_v \abs{c}_v + 27 \abs{c}_v^2\right) \leq 16 \cdot 54 \mu_v^4 .
		\end{equation*}
		With normalizations of absolute values on $k$ as in Section \ref{Definitions_strategy}, for any $x \in k^{*}$ we have $\log \abs{\NkQ(x)} = \sum_{v \in M_k^{\infty} }d_v \log \abs{x}_v$. 
		In particular,
		\begin{align}
			\log \abs{\NkQ (\Delta_{E})} &\leq \sum_{v \in M_k^{\infty}} d_v \left(\log(2^4 \cdot 54) +  4 \log \max \left\{1, \abs{a}_v, \abs{b}_v, \abs{c}_v\right\} \right) \nonumber \\
			& \leq d \log(2^4 \cdot 54) +4d \cdot h^{\infty}(a,b,c). \label{eqninfpart}
		\end{align}
		If $a,b \in \Ok$, then we have $h(E) = h(a,b,c) = h^{\infty}(a,b,c)$ which proves the final statement.
	\end{proof}
	For finding the right model in the general case, the following simple lemma will be useful.
	\begin{lemma}
		\label{lemma_alphas}
		Let $m \in \N$, $\alpha_1, \ldots, \alpha_m \in \Z$ and $n_1, \ldots, n_m \in \N = \{1, 2, 3 \ldots \}$. Then the following holds
		\begin{equation*}
			\lceil \max \left\{ \alpha_1/n_1, \ldots, \alpha_m/n_m \right\} \rceil \leq \max \left\{ 0, \alpha_1, \ldots, \alpha_m \right\} .
		\end{equation*}
	\end{lemma}
	\begin{proof}
		We separate the proof into two cases, depending on whether all $\alpha_i$-s are negative or not. If $\alpha_i < 0$ for all $i \in \{1, \ldots m\}$ then the left hand side is at most zero. Meanwhile, the right hand side is exactly zero, so the claim follows in this case.
		
		If there exists $j \in \{1, \ldots m \}$ such that $\alpha_j \geq 0$, we can define $\mathcal{I} := \{i \in \{1, \ldots m\}  : \alpha_i \geq 0\} \neq \emptyset$. With this we have that 
		\begin{equation*}
			\max \left\{ 0, \alpha_1, \ldots, \alpha_m \right\} =  \max \{ \alpha_i : i \in \mathcal{I} \}
		\end{equation*}
		and
		\begin{equation*}
			 \lceil \max \left\{ \alpha_1/n_1, \ldots, \alpha_m/n_m \right\} \rceil \leq \lceil \max \{\alpha_i/n_i : i \in \mathcal{I}\} \rceil .
		\end{equation*}
		 For all $i \in \mathcal{I}$ we have $\alpha_i/n_i \leq \alpha_i$ and so $\max \{\alpha_i/n_i : i \in \mathcal{I}\} \leq \max  \{ \alpha_i : i \in \mathcal{I} \}$. Since the right hand side is an integer, this implies  $\lceil \max \{\alpha_i/n_i : i \in \mathcal{I}\}\rceil \leq \max  \{ \alpha_i : i \in \mathcal{I} \}$.
	\end{proof}
	\begin{lemma}
	\label{model}
	Let $E/k$ be an elliptic curve as in \eqref{weierstrass_equation}.
	There exists an elliptic curve $E' : y^2 = x^3 +a'x^2+b'x+c'$ isomorphic to $E$ over $k$ such that $a',b',c' \in \Ok$ and the discriminant $\Delta_{E'}$ of $E'$ satisfies the following:
	\begin{equation}
		\log \abs{\NkQ(\Delta_{E'})}   \leq d \log(2^4 \cdot 54) + 12 \log m_k + 12dh(E),
	\end{equation}
	where $m_k$ is the Minkowski constant of $k$,  that is $m_k = \frac{d!}{d^d} \left(\frac{4}{\pi}\right)^s \abs{\Delta_k}^{1/2}$, where $d = [k: \Q]$, $s$ is the number of complex embeddings of $k$ up to complex conjugation and $\Delta_k$ is the discriminant of $k$.
	\end{lemma}
	\begin{proof}
		As $a,b$ and $c$ are not necessarily in $\Ok$, we need to find a suitable isomorphism of $E$ that will give us a Weierstrass equation with coefficients in $\Ok$, but the norm of the discriminant not being too big. This means that we are searching for $u \in k^*$ such that $a':=au^{2}$, $b':=bu^{4}$ and $c' = u^6 c$  satisfy $\abs{a'}_{\p}, \abs{b'}_{\p}, \abs{c'}_{\p} \leq 1$ for all $\p \in M_k^0$. This is equivalent to $\nu_{\p}(u) \geq \max \left\{ -\nu_{\p}(a)/2, -\nu_{\p}(b)/4, - \nu_{\p}(c)/6\right\}$ for all $\p$. Let us now set $\eta_{\p} := \lceil \max \left\{ -\nu_{\p}(a)/2, -\nu_{\p}(b)/4,  -\nu_{\p}(c)/6\right\} \rceil \in \Z$ for all $\p \in M_k^0$. We have $\eta_{\p} = 0$ for all but at most finitely many $\p$, so we can define a fractional ideal of $\Ok$ with \[I := \prod_{\p \in M_k^0} \p^{\eta_{\p}}.\]
		Now for any $u \in I \setminus \{0\}$ we will have $\nu_{\p}(u) \geq \eta_{\p}$, so $a',b'$ and $c'$ defined as above will be in $\Ok$. Therefore, we just need to find such $u$ with a small enough norm.
		
		By the geometry of numbers, there exists $u \in I \setminus \left\{0\right\}$ with $\abs{\NkQ(u)} \leq m_k \NkQideal(I)$. By the definition of $I$ and multiplicity of the ideal norm we have
		\begin{align*}
			\log \NkQideal(I) &= \sum_{\p \in M_k^0} \eta_{\p} \log \NkQideal(\p) = \sum_{\p \in M_k^0} d_\p \log \left(p(\p)^{\eta_{\p}/e(\p)}\right).
		\end{align*}
		
		Taking $m=3, \alpha_1 =- \nu_{\p}(a), \alpha_2 =  - \nu_{\p}(b), \alpha_3 =  - \nu_{\p}(c)$ and $n_1=2, n_2=4, n_3=6$ in Lemma \ref{lemma_alphas} it follows that $ \eta_{\p} \leq \max \left\{ 0, - \nu_{\p}(a), - \nu_{\p}(b), - \nu_{\p}(c) \right\}.$
		
		With this we have
		\begin{align*}
			\log \NkQideal(I) &\leq \sum_{\p \in M_k^0} d_{\p} \log p(\p)^{\max \left\{ 0, - \nu_{\p}(a), - \nu_{\p}(b),  - \nu_{\p}(c) \right\} / e(\p)} \\
				&= d \cdot h^0 (a,b,c).
		\end{align*}
		 Together with Lemma \ref{model_lemma1} this gives
		\begin{align*}
			\log \abs{\NkQ(\Delta_{E'})} &= \log \abs{\NkQ(\Delta_{E})}  + 12 \log \abs{\NkQ(u)} \\
				& \leq d \log(2^4 \cdot 54) +4d  h^{\infty}(a,b,c) + 12 \log m_k + 12 \log \NkQideal(I) \\
				& \leq  d \log(2^4 \cdot 54) +4d  h^{\infty}(a,b,c) + 12 \log m_k + 12d  h^0 (a,b,c) \\
				& \leq  d \log(2^4 \cdot 54) + 12 \log m_k + 12d h(E). \qedhere
		\end{align*}
	\end{proof}
	As in Theorem \ref{main_thm_can}, from now on we assume that $B$ is a real number with $B \geq e^{\max \left\{h(E), e\right\} }$. Taking the integral model $E' \cong E$ given by Lemma \ref{model}, we get that $\D_{E/k}$ divides $\Delta_{E'} \Ok$. With our assumption on $B$ this implies that there exists a constant $c_1(k)>0 $ such that
	\begin{equation}
		\label{norm_bound_logB}
		\log \NkQideal (\D_{E/k}) \leq \log  \abs{\NkQ(\Delta_{E'}) } \leq c_1(k) \log B ,
	\end{equation}
	namely we can take $c_1(k):= \frac{1}{e}(d \log(2^4 \cdot 54) + 12\log C_k)+12d$. 
	\begin{remark}
		\label{remark_integral}
		Lemma \ref{model_lemma1} shows that if $a,b$ and $c$ are already in $\Ok$, then $\log \abs{\NkQ(\Delta_{E})} \leq d \log(2^4 \cdot 54) +3d \cdot h(E)$. In that case we would get a corresponding bound
		\begin{equation}
			\label{norm_bound_logB_integral}
			\NkQideal (\D_{E/k}) \leq \tilde{c}_1(d) \log B,
		\end{equation}
		with $\tilde{c}_1(d) := \left(\frac{1}{e} \log(2^4 \cdot 54) + 4\right)d$ depending only on $d = [k: \Q]$. This is not without interest because if $E$ and $k$ are as in Theorem \ref{main_thm_integral}, then $c_1(k)$ is the only place where the dependence on $\Delta_k$ appears.
	\end{remark}
	\begin{remark}
		Let $E' : y^2 = x^3 + a'x^2+b'x+c' =: G(x)$ be the elliptic curve isomorphic given by Lemma \ref{model}. The canonical height of points is preserved under the isomorphism, so $\mathcal{N}_E(B) = \mathcal{N}_{E'}(B)$ for all $B \in \R$. The isomorphism also preserves all other objects associated to $E$ that we will be working with (rank, size of the torsion subgroup, minimal discriminant, conductor).
	\end{remark}
	As was done by Naccarato in \cite{Naccarato:2021Counting}, we will use the following result by Petsche, which is Theorem 2 in \cite{Petsche:2006Small}:
		\begin{theorem}
		\label{petsche}
		Let k be a number field of degree $d = [k:\Q]$ and take $c_2(d) = 10^{15} d^3$ and $c_3(d) = 104613 d $. If $E/k$ is an elliptic curve, then for all nontorsion points $P \in E(k)$ we have
		\begin{equation}
			\label{petsche_bound}
			\hat{h}(P) \geq \frac{\log \NkQideal (\D_{E/k})}{c_2(d) \sigma_{E/k}^6 \log^2(c_3(d)  \sigma_{E/k}^2)}.
		\end{equation}
	\end{theorem}
	Firstly, we notice that for this to be a nontrivial lower bound on $\hat{h}(P)$ we need $\NkQideal(\D_{E/k}) > 1$, that is the curve $E$ needs to have bad reduction with respect to at least one finite place of $k$. This is true for all but finitely many isomorphism classes of elliptic curves over $k$, by the theorem of Shafarevich (see for example Theorem IX.6.1 in \cite{Silverman:2009The-arithmetic}). Let us concentrate first on this (most interesting) case.
	\begin{lemma}
		\label{lemma_NB_first_bound}
		If $E/k$ is an elliptic curve as in \eqref{weierstrass_equation} with at least one place of bad reduction over $k$ and $B$ a real number such that $\log B \geq \max \{h(E),e\}$, then
		\[ \mathcal{N}_E (B) \leq  \abs{\Ektors} \left(c_4(k) \frac{\log B}{ \log \NkQideal(\F_{E/k})} \right)^{\frac{7}{2}\rank E(k)} . \]
	\end{lemma}
	\begin{proof}
		Because we have assumed that $E$ has at least one place of bad reduction over $k$, we know that $\NkQideal(\D_{E/k}) \geq \NkQideal(\F_{E/k}) \geq 2$. Moreover, we notice that $\log^2(c_3(d)  \sigma_{E/k}^2) \leq \sqrt{ c_3(d)} \sigma_{E/k}$, because $\sqrt{c_3(d) } \sigma > 100$ and $\log^2(x^2) \leq x$ for all $x \geq 75$. So, Theorem \ref{petsche} implies that for all nontorsion $P \in E(k)$ we have
		\begin{equation}
			\hat{h}(P)  \geq \frac{\log \NkQideal(\D_{E/k})}{c_5(d) \sigma_{E/k}^7} = \frac{\left(\log \NkQideal(\F_{E/k})\right)^7}{c_5(d) \left(\log \NkQideal(\D_{E/k})\right)^6},
		\end{equation}
		where $c_5(d)=c_2(d) \sqrt{c_3(d)}$. Moreover, using \eqref{norm_bound_logB}, we can further bound this as
		\begin{align*}
			\hat{h}(P) \geq \frac{\left(\log \NkQideal(\F_{E/k})\right)^7}{c_5(d) c_1(k)^6 (\log B)^6} =: \rho_1 > 0.
		\end{align*}
		As remarked in the previous section, this allows us to use Lemma \ref{balls} with $R = \sqrt{\log B}$ and $\rho = \sqrt{\rho_1}$ to bound $\mathcal{N}_E (B)$ as
		\begin{align*}
			\mathcal{N}_E (B) \leq \abs{\Ektors} \left(1 + 2c_5(d)^{\frac{1}{2}} c_1(k)^3 \left(\frac{\log B}{\log \NkQideal(\F_{E/k})}\right)^{\frac{7}{2}}\right)^{\rank E(k)}.
		\end{align*}
		Because $\log \NkQideal(\F_{E/k}) \leq \log \NkQideal(\D_{E/k})$, again using \eqref{norm_bound_logB} we have $c_1(k) \frac{\log B}{\log \NkQideal(\F_{E/k})} \geq 1$, which means we can  bound $\mathcal{N}_E (B)$ further as follows
		\begin{align}
			\mathcal{N}_E (B) \leq \abs{ \Ektors} \left( c_4(k)  \frac{\log B}{\log \NkQideal(\F_{E/k})} \right)^{\frac{7}{2}\rank E(k)}, \label{eq4}
		\end{align}
		with $c_4(k) = \left(c_1(k)^{\frac{7}{2}}+2 c_1(k)^3 \sqrt{c_5(d)}\right)^{\frac{2}{7}}$.
	\end{proof}
	If $E$ is a curve with everywhere good reduction over $k$, then Petsche's bound becomes trivial. But we can use a different lower bound for $\hat{h}(P)$, due to David \cite{David:1997Points}, Corollaire 1.4.
	\begin{theorem}
		\label{lemma_david}
		Let $E/k$ be an elliptic curve as in \eqref{weierstrass_equation}, $j_E$ the $j$-invariant of $E$ and $J:=\max\left\{h(j_E),1\right\}$. If $j_E \in \Ok$, then there exists an effective constant $c_6>0$ such that for any nontorsion $P \in E(k)$ we have
		\[ \hat{h}(P) \geq \frac{c_6 \cdot J^3}{d^3  \left(J+\log d\right)^2}.\]
	\end{theorem}
	In an analogous way to Lemma \ref{lemma_NB_first_bound}, we can find a similar bound for $\mathcal{N}_E (B)$ in this case.
	\begin{lemma}
		\label{lemma_evr_good_red_NB}
		Let $E/k$ be an elliptic curve as in \eqref{weierstrass_equation} with good reduction at all finite places of $k$. Then
		\[ \mathcal{N}_E (B) \leq \abs{\Ektors} \left(c_7(d)\log B \right)^{\frac{1}{2} \rank E(k)}. \]
	\end{lemma}	
	\begin{proof}
		Because $E$ has everywhere good reduction, we know that $j_E \in \Ok$. Therefore, if $P \in E(k)$ is a nontorsion point, applying Theorem \ref{lemma_david} results in 
		\[ \hat{h}(P) \geq c_6 \frac{J^3}{d^3 (J+\log d)^2} \geq \frac{c_6}{d^3(1+\log d)^2} =: \rho_2 > 0,\]
		because $J = \max \{ h(j_E),1\} \geq 1$ by definition. Using the same strategy as in the proof of  Lemma \ref{lemma_NB_first_bound}, we apply Lemma \ref{balls} with $R = \sqrt{\log B}$ and $\rho = \sqrt{\rho_2}$ to get
		\begin{align*}
			\mathcal{N}_E (B) &\leq \abs{\Ektors} \left(1 + 2 \frac{d^{3/2}(1+\log d)}{c_6^{1/2}} \sqrt{\log B}\right)^{\rank E(k)} \\
				& \leq \abs{\Ektors} \left(c_7(d) \log B \right)^{\frac{1}{2} \rank E(k)},
		\end{align*}
		where $c_7(d) = \left(1 + \frac{2d^{3/2}(1 + \log d)}{c_6^{1/2}}\right)^2$, because $\log B \geq e > 1$.
	\end{proof}

	Next, we want to get an upper bound for the size of the torsion subgroup of $E(k)$ in terms of $k$ and $B$. For this we will need the following lemma.
	\begin{lemma}
		Let us put $c_8 = 6.2$.
		Let $N$ be an integer with $N\geq 16$. There exist $p, q \in \N$ prime such that $p \neq q$, $p, q \nmid N$ and $p, q \leq c_8 \log N$. 
		\label{lemma_primes}
	\end{lemma}
	\begin{proof}
		If $\pi(x) = \abs{\left\{ p \text{ prime} : p \leq x \right\}}$ and $\omega (x) = \abs{ \left\{ p \text{ prime} : p \mid x \right\}} $, we see that proving the lemma amounts to finding $c_8 >0$ such that $\pi (c_8 \log N) \geq \omega(N)+2$. Because both sides are integers it is actually enough to have $ \pi (c_8 \log N) > \omega(N)+1 $.
		
		A standard result (3.5) of Rosser and Schoenfeld \cite{Rosser:1962Approximate} tells us that $\pi(x) > \frac{x}{\log (x)}$ for all $x \geq 17$. For all $N \geq 16$ we have that $17/\log N \leq 17/\log16 <6.2$, so by taking $c_8:=6.2$ we get $\pi(c_8 \log N) > c_8 \frac{\log N}{\log(c_8 \log N)}$. Furthermore, $N \geq 16 > e^e$ implies that $\log(c_8 \log N) < (\log c_8 +1)\log \log N$. By evaluating for $c_8 = 6.2$, we see that $c_8/(1+\log c_8)>2$, which gives us $\pi(c_8 \log N) > 2 \frac{\log N}{\log \log N}$.
		
		On the other hand, by Th\'{e}or\`{e}me 11 of Robin \cite{Robin:1983Estimation}, we know that $\omega(n) \leq c_9  \frac{\log n}{\log \log n}$ for all $n \geq 3$, where $c_9 = 	1.3841$. We can therefore estimate $\omega(N)+1 < (1.4+1/e) \frac{\log N}{\log \log N} <1.8 \frac{\log N}{\log \log N}$, because $ \log N/ \log \log N \geq e$ for all $N \geq 16$.
		
		Finally, this gives us
		\begin{align*}
			\pi(c_8 \log N) >2 \frac{\log N}{\log \log N} > \omega(N)+1,
		\end{align*}
		which proves the lemma.
		\end{proof}
	In Lemmas \ref{lemma_NB_first_bound} and \ref{lemma_evr_good_red_NB} we have derived upper bounds for $\mathcal{N}_E(B)$ that contain the factor $\abs{\Ektors}$. In the next lemma we bound the torsion from above in terms of $k$ and $B$.
	\begin{lemma}
		\label{torsion_bound}
		Let $E/k$ be an elliptic curve given by \eqref{weierstrass_equation} and $B \geq e^{\max \{h(E),e\}}$ as before. Then
		\[ \abs{ \Ektors } \leq c_{10}(k) (\log B)^{2 d}. \]
	\end{lemma}
	\begin{proof}
		For a prime ideal $\p$ of $\Ok$ let $\tilde{E}_{\p}$ denote the reduction modulo $\p$ of $E$ and $\tilde{k}_{\p}$ the residue field $\Ok / \p$.
		
		If $\p$ and $\q$ are two prime ideals of $\Ok$ such that $E$ has good reduction at both of them and if $p \neq q$, where $p$ and $q$ are the unique primes of $\Z$ lying above $\p$ and $\mathfrak{q}$ respectively, then $\Ektors$ injects into $\tilde{E}_{\p}(\tilde{k}_{\p}) \times \tilde{E}_{\mathfrak{q}}(\tilde{k}_{\mathfrak{q}})$ (see, for example, remark after Theorem C.1.4. in \cite{Hindry:2000Diophantine}). The size of $ \tilde{E}_{\p}(\tilde{k}_{\p}) $ is trivially bounded by $2 \abs{\tilde{k}_{\p}}+1 = 2 p^{f(\p) }+ 1$, where $f(\p) \leq d$ is the inertia degree of $\p$ over $p$. So, in particular we have
		\begin{align}
		\label{tors_bound_eq1}
		\abs{\Ektors } \leq (2 \abs{ \tilde{k}_{\p}}  +1)\cdot (2 \abs{\tilde{k}_{\mathfrak{q}}} +1) \leq (2p^d+1)(2q^d+1).
		\end{align}
		Hence for proving the lemma it suffices to find two small enough distinct primes $p,q \in \Z_{\geq 0}$ such that there exist prime ideals of  $\Ok$ above both that do not divide $\D_{E/k}$. Because $a', b',c' \in \Ok$, from Lemma \ref{model} and \eqref{norm_bound_logB} it follows that it suffices to find primes $p$ and $q$ not dividing $\NkQ (\Delta_{E'})$ and satisfying $p,q \leq C \log \abs{\NkQ (\Delta_{E'})}$ for some $C>0$.
		
		Because we have $a', b',c' \in \Ok$ and $\Delta_{E'} = 16(a'^2 b'^2 - 4 a'^3 c' - 4 b'^3 + 18 a' b' c' - 27 c'^2) \neq 0$, we know that $\abs{\NkQ (a'^2 b'^2 - 4 a'^3 c' - 4 b'^3 + 18 a' b' c' - 27 c'^2)} \geq 1$ and consequently $\abs{\NkQ (\Delta_{E'})} \geq 16$. So we can apply Lemma \ref{lemma_primes} to find primes $p$ and $q$ with $p \neq q$, both not dividing $\abs{\NkQ (\Delta_{E'})}$ and $p, q \leq c_8 \log \abs{\NkQ (\Delta_{E'})} \leq c_8 c_1(k) \log B$. As we argued before, these satisfy the conditions needed to apply the bound \eqref{tors_bound_eq1} and because $\log B \geq e$ this results in
		\begin{align*}
			\abs{\Ektors} \leq (2 (c_8 c_1(k)\log B)^d +1)^2 \leq c_{10}(k) (\log B)^{2d},
		\end{align*}
		where $c_{10}(k):= (2 (c_7 c_1(k))^d + 1/e^d)^2$.
	\end{proof}	
	For positive real numbers $C_1, C_2$ and $B \geq e^e$ we have
	\begin{equation}
		C_1 (\log B)^{C_2} \leq B^{C_3/ \log \log B}, \text{ for } C_3 \geq \frac{1}{e}\log C_1 + \frac{4 }{e^2} C_2,
		\label{bound_logB}
	\end{equation}
	which follows from $\log B/\log \log B \geq e $ and $ \log B / (\log \log B)^2 \geq e^2/4$.
	So the previous lemma gives us
	\begin{align}
		\label{tors_bound_final}
		\abs{\Ektors} \leq B^{c_{11}(k)/ \log \log B}.
	\end{align}
	
	All of this is enough to show that if we assume that the rank $r$ of $E(k)$ is smaller than some constant, then Theorem \ref{main_thm_can} follows.
	
	\begin{prop}
		\label{lemma_bounded_rank}
		Let $c_{12}(k)>0$ be a real number. Let $E/k$ be an elliptic curve given by \eqref{weierstrass_equation} such that $\rank E(k)$ is at most $c_{12}(k)$. Then $\mathcal{N}_E(B) \leq B^{c_{13}(k)/\log \log B}$
	\end{prop}
	\begin{proof}
		Let us first look at the case when $E$ does not have everywhere good reduction. This means that $\NkQideal (\F_{E/k}) \geq 2$, so  Lemma \ref{lemma_NB_first_bound} and \eqref{tors_bound_final} together imply 
		\begin{align*}
			\mathcal{N}_E (B) &\leq B^{c_{11}(k)/ \log \log B} \left( c_4(k)  \frac{\log B}{\log \NkQideal(\F_{E/k})} \right)^{\frac{7}{2} c_{12}(k)} \\
			&  \leq B^{c_{11}(k)/ \log \log B} \left( \frac{c_4(k)}{\log 2} \log B\right)^{\frac{7}{2} c_{12}(k)} \label{eq_bounded_rank_1}\\
			& \leq B^{c_{13}(k) / \log \log B}.
		\end{align*}
		Similarly, in the case of everywhere good reduction using Lemma \ref{lemma_evr_good_red_NB} we get
		\begin{equation*}
			\mathcal{N}_E (B) \leq  B^{c_{11}(k)/ \log \log B} \left(c_7(d)\log B \right)^{\frac{1}{2} c_{12}(k)} \leq    B^{\tilde{c}_{13}(k) / \log \log B}. \qedhere
		\end{equation*}
	\end{proof}
	
	\section{Bounding the rank} \label{Bounding_rank}
	Let $E/k$ be an elliptic curve as in \eqref{weierstrass_equation}.
	In this section we need to make further assumptions about $E$, to get a suitable bound on the rank. Therefore, we assume that $E$ is as in the statement of Theorem \ref{main_thm_can}, i.e. it has at least one point of exact order $\ell$ defined over $k$, where $\ell$ is a fixed prime integer. We will look at the field extension $K_\ell := k(E[\ell])$, so $E$ will have full $\ell$-torsion over $K_\ell$. First we analyze the extensions $K_\ell \supset k$. We start out with the simplest example, when $\ell = 2$.
	\begin{remark}
		\label{remark_2torsion}
		Let $E$ be as in Theorem \ref{main_thm_can} for $\ell=2$, so $E$ has at least  one  (nontrivial) $k$-rational point of order two. This means that there exists $x_0 \in k$ such that $F(x_0)=0$. Therefore, we have $F(x) = (x-x_0)F_1(x)$, for some $F_1 \in k[x]$ of degree 2. The extension $K_2 = k(E[2])$ is generated by roots of the polynomial $F$. But, as we have assumed $x_0 \in k$, it is generated over $k$ by roots of $F_1$, so $[K_2:k] \leq 2$. Thus, we have either $K_2=k$ or $[K_2:k] = 2$, in which case $\Gal(K_2/k) \cong \Z / 2\Z$.
	\end{remark}
	Let $\xi_\ell$ denote a primitive $\ell$-th root of unity. For $\ell = 2$ we trivially have that $\xi_\ell \in k$. However, for $\ell \geq 3$ we will need to look at $k' := k(\xi_\ell)$. From the properties of the Weil pairing we know that $K_\ell$ contains $\ell$-th roots of unity, so we have $k \subset k' \subset K_\ell$. First we investigate the extension $k' \subset K_\ell$; our aim is to show that this Galois extension is either trivial or cyclic of order $\ell$.
	
	We first do this for $\ell = 3$, in which case we can show this by finding a concrete family of curves with a $k$-rational point of exact order 3. In this case we have $k' = k(\xi_3) = k(\sqrt{-3})$.
	\begin{lemma}
		\label{lemma_3torsion}
		Let $E$ be an elliptic curve given by \eqref{weierstrass_equation} defined over a number field $k$. Suppose $E(k) \cap E[3] \neq \{O\}$ and that $k$ contains $\sqrt{-3}$. Then $k(E[3])/k$ is a Galois extension with Galois group either trivial or cyclic of order 3.
	\end{lemma}
	\begin{proof}
		As we have assumed $E(k) \cap E[3] \neq \{O\}$, there is a $P = (x_0,y_0) \in E(k) \cap E[3]$ for some $x_0,y_0 \in k$. Then we know due to Corollary 1 in \cite{Bekker:2021Versal} that there exist unique $A_1,A_2 \in k$ such that
		\begin{equation}
			\label{eq_weirstrass_3_torsion}
			F(x) = (x-x_0)^3 + (A_1 x + A_2)^2
		\end{equation}
		and $y_0 = A_1 x_0 +A_2$.  From this we have that the $3$-division polynomial of $E$ is $\psi_3(x) = (x-x_0) \Psi(x)$, where $\Psi \in k[x]$ is given by
		\begin{equation*}
			\Psi (x)= 3x^3 + (4A_1^2 - 9x_0)x^2 + (4A_1^2x_0 + 12A_1 A_2 + 9x_0^2)x + 4A_1^2 x_0^2 + 12A_1 A_2 x_0 - 3x_0^3 + 12A_2^2.
		\end{equation*}
		The discriminant of $\Psi$ is 
		\begin{align*}
			\Disc(\Psi) &= -48 (A_1 x_0 +A_2)^2 (-4 A_1^3 + 27 A_1 x_0 + 27 A_2)^2 \\
			&= \left(4\sqrt{-3}(A_1 x_0 +A_2)(-4 A_1^3 + 27 A_1 x_0 + 27 A_2)\right)^2.
		\end{align*}
		Since $\sqrt{-3} \in k$, we have $k(\Disc(\Psi)) = k(\sqrt{-3}) = k$.
		Let $L \supset k = k(\sqrt{-3})$ be the splitting field of $\Psi$. From Theorem 2.8 in \cite{Conrad:2010Galois} it follows that if $\alpha$ is a root of $\Psi$, then $L = k(\sqrt{\Disc(\Psi)}, \alpha) = k(\alpha)$. We want to show that actually $L = K_3 = k(E[3])$. 
		The field $L$ is generated by roots of $\Psi$, which are also roots of $\psi_3$, so $L \subset K_3$.
		If $Q=(x_1,y_1) \in E[3] \setminus \{O\}$, then $\psi_3(x_1) = 0$. So either $x_1= x_0 \in k \subset L$, or $\Psi(x_1)=0$. Thus, we have $x_1 \in L$. But $y_1^2 = F(x_1)$ with $F$ as in \eqref{eq_weirstrass_3_torsion}. Because $\Psi(x_1) = 0$ we have
		\begin{align*}
			-3y_1^2 &= -3 (x_1-x_0)^3 - 3(A_1 x_1+A_2)^2 + \Psi(x_1) = (A_1x _1+3A_2+2A_1 x_0)^2.
		\end{align*}
		Therefore, $y_1= \pm (A_1 x_1+3A_2 +2A_1 x_0)/\sqrt{-3} \in k(x_1) \subset L$. This shows that $K_3 = L = k(\alpha)$ for any root $\alpha$ of $\Psi$.
		
		If $\Psi$ is reducible over $k$, then since $\deg (\Psi) = 3$, there is a $\alpha \in k$ such that $\Psi (\alpha) = 0$. With the above, this means that $K_3 = k$ in this case.
		On the other hand, if $\Psi$ is irreducible over $k$, then $[k(\alpha):k]=3$ for any root $\alpha$ of $\Psi$. So $[K_3:k]=3$ and therefore $\Gal(K_3/k) \cong \Z / 3\Z$.
	\end{proof}
	In the next lemma we generalize the results of Remark \ref{remark_2torsion} and Lemma \ref{lemma_3torsion} to arbitrary prime $\ell \in \Z$, by using properties of the Weil pairing and Galois representations.
	\begin{lemma}
		\label{lemma_ell_torsion}
		Let $E$ be an elliptic curve given by \eqref{weierstrass_equation} defined over a number field $k$ and let $\ell$ be a prime integer. Suppose $E(k) \cap E[\ell] \neq \{O\}$ and that $k$ contains a primitive root of unity of order $\ell$. Then $k(E[\ell])/k$ is a Galois extension with Galois group either trivial or cyclic of order $\ell$.
	\end{lemma}
	\begin{proof}
		Since $E(k) \cap E[\ell] \neq \{O\}$, we have a point $P \in E(k) \cap E[\ell]$ with $P \neq O$. Therefore, we can find a $\Z/\ell\Z$-basis $\{P,Q\}$ for $E[\ell] \cong (\Z/\ell \Z)^2$, for some $O \neq Q \in E[\ell]$. 
		For each $\sigma \in \Gal (k(E[\ell])/k) $, since $P \in E(k)$, we have that $\sigma (P,Q) = (\sigma(P), \sigma(Q)) = (P, \alpha_\sigma P + \beta_\sigma Q)$ for some $\alpha_\sigma, \beta_\sigma \in \Z/ \ell \Z$.
		Let $\mu_\ell$ denote the group of $\ell$-th roots of unity and $e_\ell : E[\ell] \times E[\ell] \rightarrow \mu_\ell$ the Weil pairing. By the assumption $\mu_\ell \subset k$ and so $\sigma(e_\ell (P,Q)) = e_\ell (P,Q)$. From the standard properties of the Weil pairing (Proposition III.8.1. in \cite{Silverman:2009The-arithmetic}), it follows that
		\begin{align*}
			e_\ell(P,Q) &= \sigma(e_\ell (P,Q)) = e_\ell (\sigma(P), \sigma(Q)) \\
			&= e_\ell (P, \alpha_\sigma P + \beta_\sigma Q) \\
			& = e_\ell(P, \alpha_\sigma P) e_\ell (P,\beta_\sigma Q) = 1 \cdot e_\ell (P,Q)^{\beta_\sigma}.
		\end{align*}
		The Weil pairing is bilinear, alternating and nondegenerate and $\{P,Q\} $ is a basis, thus $e_\ell (P,Q) \neq 1$. So the above implies $\beta_\sigma = 1$ for every $\sigma \in \Gal (k(E[\ell])/k)$.
		Therefore, we can consider the Galois representation $\rho: \Gal (k(E[\ell])) \rightarrow \mathrm{GL}_2 (\Z / \ell \Z)$ given by
		\begin{align*}
			\rho(\sigma) =
			\begin{pmatrix}
				1 & \alpha_\sigma \\
				0 & 1
			\end{pmatrix} .
		\end{align*}
		This gives us $\Gal (k(E[\ell])) \cong \mathrm{Im} (\rho)$, which is a subgroup of
		\begin{align}
			H_\ell :=
			\left\{
			\begin{pmatrix}
				1 & \star \\
				0 & 1
			\end{pmatrix}
			\in \mathrm{GL}_2(\Z/\ell \Z)
			\right\}.
		\end{align}
		As $\abs{H_\ell} = \ell$ and $\ell$ is prime, any subgroup of $H_\ell$ must be cyclic of order either 1 or $\ell$.
	\end{proof}
	
	Because $E$ has full $\ell$-torsion over $K_\ell=k(E[\ell])$, we can make use of the following lemma, which appears as Exercise 8.1 in \cite{Silverman:2009The-arithmetic} and proof of which can be found as Theorem 2 in \cite{Ooe:1989On-the-Mordell-Weil}.
	\begin{lemma}
		\label{rank_bound_prop}
		Let $E/K$ be an elliptic curve, let $\ell \geq 2$ be an integer, let $Cl_K$ be the ideal class group of $K$, and let $S_\ell \subset M_K$ be the set consisting of places of bad reduction of $E/K$, finite places of $K$ above $\ell$, and all infinite places of $K$, that is:
		\begin{equation*}
			S_\ell = \left\{ v \in M_K ^ 0 : E \text{ has bad reduction at } v \right\} \cup \left\{ v \in M_K ^0 : v(\ell) \neq 0\right\} \cup M_K^{\infty}.
		\end{equation*}
		Assume that $E[\ell] \subset E(K)$. Then
		\begin{equation*}
		\rank _{\Z / \ell \Z} E(K)/\ell E(K) \leq 2 \: \abs{S_\ell} + 2 \: \rank_{\Z / \ell \Z} Cl_K[\ell].
		\end{equation*}
	\end{lemma}
	For any number field $L$ and an ideal $I$ in $\mathcal{O}_L$ let us denote by
	\begin{align*}
		\omega_L(I) := \abs{\left\{\p \text{ prime ideal in }L : \p \mid I\right\}}.
	\end{align*}
	Looking at the field $K_\ell = k(E[\ell])$ as defined above, we see that it trivially satisfies the condition in Lemma \ref{rank_bound_prop}. Let $S_\ell$ be the set defined in the same lemma. Then
	\begin{align*}
		\abs{S_\ell}  \leq \omega_{K_\ell}(\D_{E/K_\ell})+2[K_\ell : \Q] =  \omega_{K_\ell} (\D_{E/K_\ell}) + 2[K_\ell : k]d.
	\end{align*}

	From Proposition VII.5.4.(b) in \cite{Silverman:2009The-arithmetic}, it follows that if $E$ has good reduction at $\p \in M_k^0$, then $E$ also has good reduction at any $\mathfrak{P}\in M_{K_\ell}^0$ above $\p$.
	This shows that
	\[
	\omega_{K_\ell}(\D_{E/K_\ell}) \leq [K_\ell:k] \omega_k(\D_{E/k}),
	\]
	so we can further bound the size of $S_\ell$ as
	\begin{align*}
		\abs{S_\ell} \leq [K_\ell : k] \left(\omega_k (\D_{E/k}) +2d\right).
	\end{align*}
	Using the structure theorem for finitely generated groups we can see that $\rank_{\Z / \ell \Z} E(K_\ell) / \ell E(K_\ell) \geq \rank E(K_\ell)$, which gives the following:
	\begin{align}
		\begin{split}
			\label{rank1}
			\rank E(K_\ell) & \leq \rank_{\Z / \ell \Z} E(K_\ell)/\ell E(K_\ell) \leq 2\left(\abs{S_\ell} + \rank_{\Z / \ell \Z}Cl_K[\ell]\right) \\
			& \leq 2\left( [K_\ell : k] \left(\omega_k (\D_{E/k	}) +2d\right) + \rank_{\Z / \ell \Z}Cl_{K_\ell}[\ell]\right).
		\end{split}	
	\end{align}
	Because the rank can not decrease when passing to the bigger field, this gives an upper bound on $r :=\rank E(k) \leq \rank E(K_\ell)$.
	
	We use the following lemma which is a special case of Theorem 2.2 of Rosen \cite{Rosen:2014Class}, itself a consequence of previous work by Cornell in \cite{Cornell:1983Relative}, which relates $Cl_{K_\ell} [\ell]$ back to $Cl_k [\ell]$ in some specific cases. It uses the relative genus theory, which can be seen as a generalization of the classical Gauss genus theory. We also note that $Cl_L/\ell Cl_L \cong Cl_L[\ell]$ for any number field $L$ and prime $\ell \in \Z$.
	
	\begin{lemma}
		\label{lemma_rosen}
		Let $\ell \in \Z$ be a prime and $K/k$ a Galois extension of number fields with $\Gal(K/k)$ cyclic of order $\ell$.
		Let $t_{K/k}$ be the number of primes (finite or infinite) in $k$ which ramify in $K$. Then
		\begin{equation}
			\rank_{\Z / \ell \Z}Cl_K/\ell Cl_K \leq \ell \left( \rank_{\Z / \ell \Z}Cl_k/\ell Cl_k + \max \left\{ 0, t_{K/k}-1 \right\} \right).
		\end{equation}
	\end{lemma}
	Let us recall the definition of a ramified infinite prime in an extension $k \subset K$. Let $v \in M_k^\infty$ and $\sigma :k
	 \hookrightarrow \C$ be a corresponding embedding. We say that $v$ ramifies in $K$ if $\sigma$ is a real embedding and there exists an extension $\bar{\sigma}: K \hookrightarrow \C$ of $\sigma$ (i.e. $\bar{\sigma} \vert_k = \sigma$) with $ \mathrm{Im}(\bar{\sigma}) \not\subset \R$.
		
	From Remark \ref{remark_2torsion} we know that if $E(k) \cap E[2] \neq \{O\}$, we have $[K_2:k] \leq 2$. In the case when this degree is exactly $2$ we can apply Lemma \ref{lemma_rosen} to get
	\begin{equation}
		\label{eq_rosen_2}
		\rank_{\Z / 2 \Z} Cl_{K_2} [2] \leq 2 \left(\rank_{\Z / 2 \Z} Cl_k [2]  + t_{K_2/k}\right).
	\end{equation}
	Similarly, if $\ell=3$ and $E(k) \cap E[3] \neq \{O\}$,  in Lemma \ref{lemma_3torsion} we have seen that we have either $K_3=k'$ or $\Gal(K_3/k') \cong \Z / 3\Z$, where $k' = k(\sqrt{-3})$. In the second case Lemma \ref{lemma_rosen} with $\ell=3$ for the extension $K_3 \supset k'$ gives
	\begin{align}
		\label{eq_rosen_3torsion_general}
		\rank_{\Z / 3 \Z} Cl_{K_3} [3] \leq 3 \left(\rank_{\Z /3 \Z} Cl_{k'} [3]  + t_{K_3/k'}\right).
	\end{align}
	
	From Lemma \ref{lemma_ell_torsion} we get an equivalent bound for any prime $\ell \in \Z$.
	If $E(k) \cap E[\ell] \neq \{O\}$, then Lemma \ref{lemma_ell_torsion} applied to $k' = k(\xi_\ell)$ implies that either $K_\ell = k'$ or $\Gal(K_\ell / k') \cong \Z / \ell \Z$. In the second case Lemma \ref{lemma_rosen} for the extension $K_\ell \supset k'$ gives
	\begin{align}
		\label{eq_rosen_ell_torsion}
		\rank_{\Z / \ell \Z} Cl_{K_\ell} [\ell] \leq \ell \left(\rank_{\Z /\ell \Z} Cl_{k'} [\ell]  + t_{K_\ell/k'}\right).
	\end{align}
	Note that the bounds \eqref{eq_rosen_2}--\eqref{eq_rosen_ell_torsion} are also trivially satisfied in the case when $K_\ell = k'$, so we can proceed without distinguishing these two cases. Therefore we can bound the rank as
	\begin{align}
		r &\leq 2 \left([K_\ell:k]\left(\omega_k(\D_{E/k})+2d\right)+ \ell \left(\rank_{\Z /\ell \Z} Cl_{k'} [\ell]  + t_{K_\ell/k'}\right)\right) \nonumber \\
		& \leq 2 \ell \left([k':k]\left(\omega_k(\D_{E/k})+2d\right) + \rank_{\Z /\ell \Z} Cl_{k'} [\ell]  + t_{K_\ell/k'} \right).
		\label{rank_bound_together_general}
	\end{align}
	If we assume that $k$ contains a primitive $\ell$-th root of unity, that is $k'=k$, the bound above becomes
	\begin{align}
		\label{eq_rank_bound_ell}
		r \leq 2 \ell \left(\omega_k(\D_{E/k})+2d + \rank_{\Z /\ell \Z} Cl_{k} [\ell]  + t_{K_\ell/k} \right).
	\end{align}
		
	We want to bound $t_{K_\ell/k}$ as defined in Lemma \ref{lemma_rosen} in terms of $E$ and $k$. So we want to bound the number of finite and infinite primes of $k$ that ramify in $K_\ell=k(E[\ell])$. For the infinite case, we have that the number of ramified primes is trivially bounded by the number of all embeddings of $k$ into $\C$, so it is at most $d=[k:\Q]$. The following lemma deals with the finite primes. It is  essentially the converse of the N\'{e}ron--Ogg--Shafarevich criterion, just stated for global fields.
	\begin{lemma}
		\label{lemma_ramified_primes}
		Let $\ell \geq 2$ be an integer and $E$ an elliptic curve over a number field $k$ with the minimal discriminant $\D_{E/k	}$. Let $t_0$ denote the number of finite primes of $k$ that ramify in $K:=k(E[\ell])$. Then
		\begin{equation*}
			t_0 \leq \omega_k (\ell \D_{E/k}).
		\end{equation*}
	\end{lemma}
	\begin{proof}
		We have to show that if a prime $\p \subset \Ok$ ramifies in $K$, then $\p \mid \ell$ or $\p \mid \D_{E/k}$, or equivalently: if $\p \nmid \ell$ and $\p \nmid \D_{E/k}$, then $\p$ is unramified in $K$. So let $\p$ be a maximal ideal of $\Ok$ such that $v_\p(\ell) = 0$ and $\p \nmid \D_{E/k}$, i.e. $E$ has good reduction at $\p$.  Let $\mathfrak{P} \subset \OK$ be any prime lying above $\p$ and $e_{\mathfrak{P}}$ its ramification index. Let $k_\p, K_{\mathfrak{P}}$ be the corresponding completions, so $k_\p \subset K_{\mathfrak{P}}$. If $\hat{\p}, \hat{\mathfrak{P}}$ are the unique maximal ideals of $k_\p,K_{\mathfrak{P}}$ respectively, then we know that the ramification index of $\hat{\mathfrak{P}}$ over $\hat{\p}$ is equal to $e_{\mathfrak{P}}$. Therefore it is enough to show that $K_{\mathfrak{P}} \supset k_\p$ is unramified.
		
		Because $K/k$ is a Galois extension, we know that $K_{\mathfrak{P}} / k_\p$ is as well, so we can look at $\Gal(K_{\mathfrak{P}} / k_\p)$. Let us denote by $\tilde{F_\p}$ and $\tilde{F_{\mathfrak{P}}}$ the corresponding residue fields and by $R_\p$ and $R_{\mathfrak{P}}$ the rings of integers of $k_\p$ and $K_{\mathfrak{P}}$, respectively. Let $\Inertiap$ be the inertia group, so $\Inertiap = \left\{\sigma \in \Gal(K_{\mathfrak{P}} / k_\p) : \bar{\sigma} = id_{\tilde{F_{\mathfrak{P}}}} \right\}$, where $\sigma \in \Gal(\tilde{F_{\mathfrak{P}}} / \tilde{F_\p})$ is the field isomorphism induced by $\sigma$, i.e. $\sigma (x) = \sigma (x) + \sigma (\mathfrak{P})$ for $x \in R_{\mathfrak{P}}$. We know that the fixed field of $\Inertiap$ is the maximal unramified extension of $k_\p$ in $K_{\mathfrak{P}}$. So if we denote this field by $L$, we want to show that $L \supset K_{\mathfrak{P}}$, which amounts to showing that $\Inertiap$ acts as the identity on $K_{\mathfrak{P}}$.
		
		Because $K = k(E[\ell])/k$ is a finite extension, we have that
		\begin{equation}
			K_{\mathfrak{P}} = (k(E[\ell]))_{\mathfrak{P}} = k_\p (E[\ell]).
			\label{kP(Em))}
		\end{equation}
		Now let $\sigma$ be in $\Inertiap \subset \Gal({K_{\mathfrak{P}}/k_\p})$, so $\Inertiap$ acts trivially on $k_\p$. Because of \eqref{kP(Em))}, if we show that $\Inertiap$ acts trivially on $E[\ell]$, we can conclude that it does so on the whole $K_{\mathfrak{P}}$.
		
		Since we have assumed that $v_\p(\ell)=0$ and $E$ has good reduction at $\p$, it follows that also $v_{\hat{\mathfrak{P}}}(\ell)=0$ and $E$ has good reduction at $\hat{\mathfrak{P}}$. With this we can apply Proposition VII.4.1.(a) in \cite{Silverman:2009The-arithmetic}, which implies that $E[\ell]$ is unramified at $\hat{\mathfrak{P}}$. This means exactly that the action of $\Inertiap$ on $E[\ell]$ is trivial.
		
	\end{proof}
	For $K_\ell = k(E[\ell])$ this gives us $t_{K_\ell / k} \leq d + \omega_k(\ell \D_{E/k	}) \leq 2d+ \omega_k(\D_{E/k})$. 
	If $k$ contains a primitive $\ell$-th root of unity, then together with \eqref{eq_rank_bound_ell} this implies
	\begin{equation}
		\label{rank_bound_together2}
		r \leq 2 \ell \left(2 \omega_k(\D_{E/k}) + 4d + \rank_{\Z / \ell \Z} Cl_k[\ell]\right).
	\end{equation}
	
	It remains to study the case when $k$ does not contain a primitive $\ell$-th root of unity. For that we need to bound the size of class group of $k'= k (\xi_\ell)$ in terms of $k$, specifically the discriminant $\Delta_{k}$ of $k$. Since this is a cyclotomic extension of $k$, this is not too hard.
	\begin{lemma}
		\label{Clk'_first_bound}
		Let $\ell \geq 3$ be a prime integer, $k$ a number field with $d:=[k:\Q]$ and $k':=k(\xi_\ell)$, where $\xi_\ell$ is an $\ell$-th root of unity. Then 
		\begin{equation*}
			\abs{Cl_{k}} \leq c_{14}(d) \ell^{\frac{\ell d}{2}} \abs{\Delta_{k}}^{\frac{\ell}{2}}.
		\end{equation*}
	\end{lemma}
	
	\begin{proof}
		Firstly, we notice that if $k' = k = \Q$, the bound is trivially satisfied. Therefore, we can assume that $[k':\Q] \geq 2$.
		Let $d'$ denote the degree of $k'$ over $\Q$.
		From Theorem 4.10 in \cite{Narkiewicz:2004Elementary}, it follows that
		\begin{align*}
			\abs{Cl_{k'}} &\leq c_{15}(d')\abs{\Delta_{k'}}^{\frac{1}{2}}\left(\log \abs{\Delta_{k'}}\right)^{d'-1}, 
		\end{align*}
		where $\Delta_{k'}$ is the discriminant of $k'$ and $c_{15}(d')$ is an effective constant. 
		For any $\epsilon>0$ and $x\geq 1$ we have that
		\begin{equation*}
			\frac{(\log x)^{d'-1}}{x^\epsilon} \leq \left(\frac{d'-1}{\epsilon e}\right)^{d'-1},
		\end{equation*}
		so by taking $\epsilon = \frac{1}{2(\ell-1)}$ and $x= \abs{\Delta_{k'}}$ we get
		\begin{equation*}
		\abs{\Delta_{k'}}^{\frac{1}{2}}\left(\log \abs{\Delta_{k'}}\right)^{d'-1} \leq c_{16}(d') \abs{\Delta_{k'}}^{\frac{\ell}{2(\ell-1)}}.
		\end{equation*}
		Let $\Delta_{k'/k} \subset \Ok$ be the relative discriminant. Then $\abs{\Delta_{k'}} = \NkQideal (\Delta_{k'/k}) \abs{\Delta_k}^{[k':k]}$. Furthermore, we have that $d'=[k':k]d \leq (\ell-1)d$ and so $\abs{\Delta_{k'}} \leq \NkQideal (\Delta_{k'/k}) \abs{\Delta_k}^{\ell-1}$. We can further bound $\abs{Cl_k'}$ as follows:
		\begin{equation*}
			\abs{Cl_{k'}} \leq c_{14}(d) \NkQideal(\Delta_{k'/k})^{\frac{\ell}{2(\ell-1)}}\abs{\Delta_{k}}^{\frac{\ell}{2}}.
		\end{equation*}
		All that is left is to bound the norm of $\Delta_{k'/k} \subset \Ok$. 
		Because $k' \supset k$ is a cyclotomic extension, from Theorem 1.1. (4) in \cite{Cohen:2003Cyclotomic} it follows that $\NkQideal (\Delta_{k'/k}) \leq \ell^{(\ell-1)d}$ and therefore
		\begin{align*}
			\abs{Cl_{k'}} & \leq c_{14}(d) \ell^{\frac{\ell d}{2}} \abs{\Delta_{k}}^{\frac{\ell}{2}}.
			\qedhere
		\end{align*}
	\end{proof}

	With $k'$ as before, it follows from the previous Lemma that if $k' \neq k$, we have
	\begin{align}
		 \rank_{\Z /\ell \Z} Cl_{k'} [\ell]  \leq \frac{1}{\log \ell}  \log \abs{Cl_{k'}} \leq c_{17} (k),
		 \label{eq_class_group_bound}
	\end{align}
	 where $c_{17}(k) = \frac{1}{\log \ell}\log\left( c_{14}(d) \ell^{\frac{\ell d}{2}} \abs{\Delta_{k}}^{\frac{\ell}{2}}\right)$.
	We can summarize these conclusion into the following Lemma.
	\begin{lemma}
		\label{rank_bound_omega}
		Let $E/k$ be an elliptic curve defined by \eqref{weierstrass_equation} and $\ell$ a prime integer. Assume further that $E(k) \cap E[\ell] \neq \{O\}$ and let $r = \rank E(k)$. Then 
		\begin{align*}
			r \leq c_{18}(k)+c_{19} \omega_k(\D_{E/k	}).
		\end{align*}
		Furthermore, if $k$ contains a primitive $\ell$-th root of unity, then $c_{18}(k) $ depends only on $d, \ell$ and $\rank_{\Z / \ell \Z} Cl_k$. Otherwise it depends on $\abs{\Delta_k}$.
	\end{lemma}
	\begin{proof}
		Let $k'=k(\xi_\ell)$ as before, with $\xi_\ell$ a primitive $\ell$-th root of unity. Then if $k'=k$, the Lemma follows from \eqref{rank_bound_together2} immediately, with $c_{18}(k) = 2\ell \left(4d + \rank_{\Z / \ell \Z} Cl_k[\ell]\right) $ and $c_{19} = 4 \ell$.
		
		Assume that $k' \neq k$. Notice that this is not possible if $\ell=2$, since $\xi_2 = -1$ is always in $k$. We have that $[k':k]\leq \ell-1$ and so \eqref{rank_bound_together_general} implies
		\begin{align*}
			r \leq  2 \ell \left((\ell-1)\left(\omega_k(\D_{E/k})+2d\right) + \rank_{\Z /\ell \Z} Cl_{k'} [\ell]  + t_{K_\ell/k'} \right).
		\end{align*}
		Above each place of $k'$ there are at most $\ell-1$ places of $k$, hence $t_{K_\ell / k'} \leq (\ell-1)t_{K_\ell/k}$. So by Lemma \ref{lemma_ramified_primes} we have $t_{K_\ell / k'} \leq (\ell-1)(d+\omega_k(\ell \D_{E/k})) \leq (\ell-1) (2d+\omega_k(\D_{E/k}))$. Combining this with \eqref{eq_class_group_bound}, we get the following bound on the rank:
		\begin{align*}
			r &\leq  2\ell\left((\ell-1)\left(\omega_k(\D_{E/k})+2d\right) + c_{17} (k)  + (\ell-1) (2d+\omega_k(\D_{E/k})) \right) \\
			 & = 2\ell \left( 2(\ell-1)\left(2d + \omega_k(\D_{E/k})\right)+ c_{17} (k)\right).
		\end{align*}
		Therefore we can take $c_{18}(k) = 2\ell\left(4(\ell-1)d+ c_{17} (k)\right)$ and $c_{19} = 4\ell (\ell-1)$.
	\end{proof}
	\begin{lemma}
		\label{lemma_omega_k}
		There exists a constant $c_{20}(d) > 0$ such that for every integral ideal $I$ of $\Ok$ with $\NkQideal (I) \geq 3$ we have
		\[ \omega_k(I) \leq c_{20}(d) \frac{\log \NkQideal (I)}{\log \log \NkQideal (I)}\]
	\end{lemma}
	\begin{proof}
		Let $I \subset \Ok$ be an ideal, set $n_I := \NkQideal (I) \geq 3$ and let $I$ factorize in $\mathcal{O}_k$ as $I = \p_1^{e_1} \cdots \p_m^{e_m}$ with $\p_1, \ldots, \p_m$ pairwise distinct. Then by definition $\omega_k(I) = m$. Because of multiplicity of the norm, we have $n_I = \NkQideal (\p_1)^{e_1} \cdots \NkQideal (\p_m)^{e_m} =  p_1^{f(\p_1)e_1} \cdots  p_m^{f(\p_m)e_m}$, where $p_i \in \Z$ is the unique prime such that $\p_i$ lies above $p_i$ and $f(\p_i)$ is the inertia degree of $\p_i$ over $p_i$. The primes $p_i$ are not necessarily pairwise distinct, but because there are at most $d=[k:\Q]$ prime ideals of $\Ok$ above any prime in $\Z$, we see that $m \leq d \cdot \omega (n_I)$.
		Because $n_I \geq 3$, we have that $\omega(n_I) \leq 1.3841 \frac{\log n_I}{\log \log n_I} $, as per  Th\'{e}or\`{e}me 11 of \cite{Robin:1983Estimation}.
		So, for $c_{20}(d):= 1.3841d$ we have
		\[\omega_k(I) \leq d \omega(n_I)  \leq  1.3841 d \frac{\log n_I}{\log \log n_I} = c_{20}(d) \frac{\log n_I}{\log \log n_I}. \qedhere \]
	\end{proof}
	Because $\Delta_{E'} \in \Ok$, we have that $\omega_k(\D_{E/k}) \leq \omega_k(\Delta_{E'} \Ok)$ and $\NkQideal (\Delta_{E'} \Ok) \geq 16 >e^e$, so the previous lemma together with the bound \eqref{norm_bound_logB} gives us that
	\begin{align*}
		\omega_k(\D_{E/k}) &\leq c_{20}(d) \frac{\log \abs{\NkQ (\Delta_{E'})} }{\log \log \abs{\NkQ (\Delta_{E'}) }} \leq c_{20}(d) \frac{c_1(k )\log B}{\log(c_1(k) \log B)} \\
												& \leq c_{21}(k) \frac{\log B}{\log \log B},
	\end{align*}
	with $c_{21}(k) := c_{20}(d) c_1(k)$; for the second inequality we have also used the fact that the function $x \mapsto x/\log x$ is increasing on $(e,\infty)$, while the third one follows from $c_1(k) \geq 1$. Combining this with the bound on the rank from Lemma \ref{rank_bound_omega}, we get
	\begin{equation}
		r \leq c_{18}(k) + c_{19} c_{21}(k)\frac{\log B}{\log \log B} \leq c_{22}(k)\frac{\log B}{\log \log B} ,
		\label{rank_bound_logB}
	\end{equation}
	with $c_{22}(k) = c_{18}(k) + c_{19} c_{21}(k)$.

	\section{Finishing the proof of Theorem \ref{main_thm_can}} \label{Finishing_proof}
	With the results of the previous section we can return to bounding $\mathcal{N}_E (B) $, with the assumptions that $E/k$ is an elliptic curve given by \eqref{weierstrass_equation} which satisfies $E(k) \cap E[\ell] \neq \{O\}$ for some fixed prime $\ell \in \Z$. We also assume $\log B \geq \max\{h(E),e\}$. For ease of notation we write $r = \rank (E(k))$.
	
	Firstly we return to a special case when $E$ has everywhere good reduction. Then Theorem \ref{main_thm_can} follows from Lemma \ref{lemma_evr_good_red_NB} and Lemma \ref{rank_bound_omega}.
	\begin{lemma}
		\label{final_proof_good_red}
		Let $E$ be an elliptic curve as above and with everywhere good reduction over $k$, while $B$ and $\mathcal{N}_E (B)$ are as before. Then
		\[\mathcal{N}_E (B) \leq B^{{c_{23}(k)}/\log \log B}.\]
	\end{lemma}
	\begin{proof}
		With the same assumptions on $E$ in Lemma \ref{lemma_evr_good_red_NB} we had
		\begin{align*}
			 \mathcal{N}_E (B) &\leq \abs{\Ektors} \left(c_7(k)\log B \right)^{r/2}.
		\end{align*}
		Additionally, we have $\abs{\Ektors} \leq B^{c_{11}(k)/ \log \log B}$ due to \eqref{tors_bound_final}. Moreover, because $E$ has everywhere good reduction over $k$, we have $\omega_k(\D_{E/k	}) = 0$, so the rank bound from Lemma \ref{rank_bound_omega} becomes $r \leq c_{18}(k)$. In total, this gives us
		\begin{align*}
			\mathcal{N}_E (B) \leq B^{c_{11}(k)/ \log \log B} \left(c_7(k)\log B \right)^{ c_{18}(k)/2} \leq B^{c_{23}(k)/\log \log B},
		\end{align*}
		where $c_{23}(k) \geq  c_{11}(k) + \frac{1}{2} c_{18}(k)(\frac{1}{e}\log c_7(k) + \frac{4}{e^2})$, due to \eqref{bound_logB}).
	\end{proof}
	With this in mind, from now on we assume that $E$ does not have good reduction at all places of $k$. So with Lemma \ref{lemma_NB_first_bound} and the bound on the torsion \eqref{tors_bound_final}) we have
	\begin{align*}
		\mathcal{N}_E (B) \leq B^{c_{11}(k)/ \log \log B} \left(c_4(k) \frac{\log B}{ \log \NkQideal(\F_{E/k})}\right)^{\frac{7}{2} r}. 
	\end{align*}
	Furthermore, using \eqref{rank_bound_logB} we see that
	\begin{equation*}
		c_4(k)^{\frac{7}{2}r} \leq c_4(k)^{\frac{7}{2} c_{22}(k) \frac{\log B}{\log \log B} } \leq B^{c_{24}(k)/ \log \log B},
	\end{equation*}
	for $c_{24}(k)= \frac{7}{2} c_{22}(k) \log c_4(k)$, and so
	\begin{equation}
		\mathcal{N}_E(B) \leq B^{c_{25}(k)/\log \log B} \left(\frac{\log B}{ \log \NkQideal(\F_{E/k})} \right)^{\frac{7}{2}r}.
	\end{equation}
	This shows that it is enough to bound 
	\begin{equation*}
	\label{N_B_bound_final}
	\left(\frac{\log B}{ \log \NkQideal(\F_{E/k})} \right)^r = \exp\left(r \log\left(\frac{\log B}{ \log \NkQideal(\F_{E/k})}\right)\right).
	\end{equation*}
	
	Because the primes of $\Ok$ dividing the minimal discriminant are precisely the same ones dividing the conductor (exactly the ones at which $E$ has bad reduction), we have that $\omega_k(\F_{E/k}) = \omega_k(\D_{E/k	})$. So Lemma \ref{rank_bound_omega} implies
	\begin{equation}
		\label{rank_bound_1}
		r \leq c_{26}(k) \omega_k (\F_{E/k}),
	\end{equation}
	for $c_{26}(k) = c_{18}(k)+c_{19} $, because $\omega_k(\F_{E/k}) \geq 1$ as we have assumed that $E$ does not have everywhere good reduction.
	
	We will address the case when $\NkQideal(\F_{E/k})$ is small (i.e. at most 15) later. Now if $\NkQideal(\F_{E/k}) \geq 3$, then due to Lemma \ref{lemma_omega_k} we have
	\begin{equation*}
		r \leq c_{26}(k) c_{20}(d) \frac{\log \NkQideal (\F_{E/k})}{\log \log \NkQideal (\F_{E/k})}.
	\end{equation*}
	Finally, to finish the proof of the theorem, we will need the following lemma:
	\begin{lemma}
		\label{lemma_final_function}
		Let $\log B \geq e$, $x \ge e$ and
		\begin{equation*}
			f(x) := \frac{x}{\log x} \log \left(\frac{\log B}{x}\right).
		\end{equation*}
		Then $f(x) \leq c_{27} \frac{\log B}{\log \log B}$, where $c_{27} = \frac{16}{e^2}$.
	\end{lemma}
	\begin{proof}
		We separate the proof into two cases, depending on whether $x < \sqrt{\log B}$ or $x \geq \sqrt{\log B}$.
		Assuming that $ x < \sqrt{\log B}$, we have that
		\begin{equation*}
			f(x) \leq \sqrt{\log B} \log \log B,
		\end{equation*}
		because $x \geq e$ implies $\log x \geq 1$. Since $\log^2 t / \sqrt{t} \leq 16e^{-2}$ for all $t > 1$, we have $\sqrt{t} \log t \leq 16e^{-2} \cdot t/\log(t)$. By setting $t = \log B$ we get
		\begin{equation*}
			f(x) \leq 16e^{-2	} \frac{\log B}{\log \log B},
		\end{equation*}
		which proves the lemma in the first case.
		
		On the other hand if $x \geq \sqrt{\log B}$, then also $\log x \geq \frac{1}{2} \log \log B$. Again because $x \geq e > 1$ we have $ \frac{x}{\log x} > 0$, so there exists $\delta > 0$ such that $\frac{x}{\log x} = \delta \frac{\log B}{\log \log B}$. Putting this into $f$ we get
		\begin{align*}
			f(x) = \delta \frac{\log B}{\log \log B} \log \left(\frac{\log \log B}{\delta \log x}\right) \leq  \delta \frac{\log B}{\log \log B} \log(\frac{2}{\delta}).
		\end{align*}
		Because $\delta \log(2/\delta) \leq 2e^{-1}$ for all $\delta >0$, this implies
		\begin{equation*}
			f(x) \leq 2e^{-1} \frac{\log B}{\log \log B}.	\qedhere
		\end{equation*}
	\end{proof}

	We this we have enough to finish the proof of the main theorem.
	\begin{proof}[Proof of Theorem \ref{main_thm_can}]
		In Lemma \ref{final_proof_good_red} we have already proven the theorem in case when $E$ has good reduction at all places of $k$. So we can assume this is not the case for $E$ and therefore $\log \NkQideal (\F_{E/k}) \geq \log 2$. Let $r$ be the rank of $E(k)$ as before. 
		With assumptions in Theorem \ref{main_thm_can}, from \eqref{N_B_bound_final} it follows that 
		\begin{align*}
			\mathcal{N}_E (B) &\leq B^{c_{25}(k)/ \log \log B} \exp\left(\frac{7}{2}r\log\left(\frac{\log B}{ \log \NkQideal(\F_{E/k})}\right)\right).
		\end{align*}
		Let us first handle the case $\log \NkQideal (\F_{E/k}) < e$. We have that $\NkQideal(\F_{E/k}) \in \N$ because $\F_{E/k}$ is an integral ideal of $\Ok$. So $\NkQideal (\F_{E/k})<e^e$ implies $\NkQideal(\F_{E/k}) \in \{2, \ldots,15\}$. From this we can easily see that $\omega (\NkQideal (\F_{E/k})) \leq 2$. Combining this with \eqref{rank_bound_1} we get $r \leq 2d c_{26}(k)$. By taking $c_{12}(k) = 2d c_{26}(k)$ in Proposition \ref{lemma_bounded_rank} the claim follows.
		
		Assuming that $\log \NkQideal (\F_{E/k}) \geq e > \log 3$ we can apply Lemma \ref{lemma_omega_k} to $I = \F_{E/k}$ and combine this with \eqref{rank_bound_1} to get
		\begin{align*}
			r \leq c_{26}(k) c_{20}(d) \frac{\log \NkQideal (\F_{E/k})}{\log \log \NkQideal (\F_{E/k})}.
		\end{align*}
		Applying Lemma \ref{lemma_final_function} to $x = \log \NkQideal (\F_{E/k}) \geq e$ we get
		\begin{align*}
			\mathcal{N}_E (B) &\leq B^{c_{25}(k)/ \log \log B} \exp\left(\frac{7}{2}c_{26}(k)  c_{20}(d) c_{27} \frac{\log B}{\log \log B}\right) \\
				& \leq B^{c_{28}(k)/\log \log B}.  \qedhere
		\end{align*}
	\end{proof}
	Finally, we notice that if $E$ is as in the statement of Theorem \ref{main_thm_integral}, the constant $c_{18}(k)$ in Lemma \ref{rank_bound_omega} does not depend on $\Delta_k$, but only on $d, \ell$ and the $\ell$-part of the class group. The only other place where the dependence on $\Delta_k$ appeared in the constants was in $c_1(k)$. So by invoking Remark \ref{remark_integral} we get the proof of Theorem \ref{main_thm_integral}.
	\begin{proof}[Proof of Theorem \ref{main_thm_integral}]
		In Remark \ref{remark_integral} we noticed that if $E$ is defined by a Weierstrass equation with coefficients in $\Ok$, we can use inequality \eqref{norm_bound_logB_integral} instead of \eqref{norm_bound_logB}. This means that we can replace $c_1(k)$ with $\tilde{c}_1(d)$ wherever it appears in the proof of Theorem \ref{main_thm_can} and preceding lemmas. The only other place where the dependence on $\Delta_k$ appears is in Lemma \ref{rank_bound_omega} in the case when $\ell \geq 3$ and $k$ does not contain a primitive $\ell$-th root of unity. Thus, with assumptions of Theorem \ref{main_thm_integral} this will give us a final constant that depends only on $d$ and $\rank_{\Z / \ell \Z} Cl_k[\ell]$.
	\end{proof}
	
	\printbibliography
\end{document}